\nonstopmode
\documentclass[reqno,10pt]{amsart}
\usepackage{latexsym}
\usepackage{fancyhdr}
\usepackage{amsmath, amssymb}
\usepackage[ansinew]{inputenc}
\usepackage{pdflscape}
\usepackage{longtable}
\usepackage{rotating}
\usepackage{verbatim}
\usepackage{tikz}
\usepackage{hyperref}
\usepackage{subfigure}
\usepackage{mathrsfs}
\usepackage{cleveref}
\usepackage{enumitem}
\makeatletter
\newcommand{\mylabel}[2]{#2\def\@currentlabel{#2}\label{#1}}
\makeatother

\usepackage{mdwlist}
\usepackage{color}

\usepackage[all]{xy}
\newdir{ >}{{}*!/-10pt/@{>}}
\newdir^{ (}{{}*!/-5pt/@^{(}}
\newdir_{ (}{{}*!/-5pt/@_{(}}

\newcounter{mnotecount}[section]

\newcommand{\rmnote}[1]{}

\DeclareFontFamily{U}{mathb}{\hyphenchar\font45}
\DeclareFontShape{U}{mathb}{m}{n}{
      <5> <6> <7> <8> <9> <10> gen * mathb
      <10.95> mathb10 <12> <14.4> <17.28> <20.74> <24.88> mathb12
      }{}
\DeclareSymbolFont{mathb}{U}{mathb}{m}{n}
\DeclareFontSubstitution{U}{mathb}{m}{n}

\let\dot\relax
\DeclareMathAccent{\dot}{0}{mathb}{"39}
\let\ddot\relax
\DeclareMathAccent{\ddot}{0}{mathb}{"3A}
\let\dddot\relax
\DeclareMathAccent{\dddot}{0}{mathb}{"3B}
\let\ddddot\relax
\DeclareMathAccent{\ddddot}{0}{mathb}{"3C}

\theoremstyle{plain}
\newtheorem*{theorem*}{Theorem}
\newtheorem{theorem}{Theorem}
\newtheorem*{lemma*}{Lemma}
\newtheorem{lemma}[theorem]{Lemma}
\newtheorem*{proposition*}{Proposition}
\newtheorem{proposition}[theorem]{Proposition}
\newtheorem*{corollary*}{Corollary}

\newtheorem*{claim*}{Claim}

\newtheorem*{conjecture*}{Conjecture}

\newtheorem*{question*}{Question}
\theoremstyle{definition}
\newtheorem*{definition*}{Definition}
\newtheorem{definition}[theorem]{Definition}
\newtheorem*{example*}{Example}
\newtheorem{example}[theorem]{Example}

\newtheorem*{algorithm*}{Algorithm}
\newtheorem*{remark*}{Remark}
\newtheorem*{remarks*}{Remarks}
\newtheorem{remark}[theorem]{Remark}

\newtheorem*{convention*}{Convention}

\numberwithin{equation}{section}

\newcommand{\al}{\alpha}
\newcommand{\be}{\beta}
\newcommand{\ga}{\gamma}
\newcommand{\de}{\delta}
\newcommand{\ep}{\epsilon}

\newcommand{\ka}{\kappa}

\newcommand{\rh}{\rho}

\newcommand{\si}{\sigma}

\newcommand{\ta}{\tau}

\newcommand{\vh}{\varphi}

\newcommand{\ps}{\psi}
\newcommand{\om}{\omega}

\newcommand{\N}{\mathbb{N}}

\newcommand{\R}{\mathbb{R}}

\newcommand{\cB}{\mathcal{B}}

\newcommand{\cE}{\mathcal{E}}

\newcommand{\fS}{\mathfrak{S}}

\newcommand{\fW}{\mathfrak{W}}

\newcommand{\p}{\partial}

\renewcommand{\o}{\circ}

\newcommand{\on}{\operatorname}

\newcommand{\sr}[1]%
{\ifmmode{}^\dagger\else${}^\dagger$\fi\ifvmode
\vbox to 0pt{\vss
 \hbox to 0pt{\hskip\hsize\hskip1em
 \vbox{\hsize3cm\raggedright\pretolerance10000
 \noindent #1\hfill}\hss}\vss}\else
 \vadjust{\vbox to0pt{\vss%
 \hbox to 0pt{\hskip\hsize\hskip1em%
 \vbox{\hsize3cm\raggedright\pretolerance10000%
 \noindent #1\hfill}\hss}\vss}}\fi%
}

\newcommand{\ol}{\overline}

\title[On the extension of Whitney ultrajets of Beurling type]
{On the extension of Whitney ultrajets of Beurling type}

\author[A.~Rainer]{Armin Rainer}

 \address{A.~Rainer: 
 Fakult\"at f\"ur Mathematik, Universit\"at Wien, 
 Oskar-Morgenstern-Platz~1, A-1090 Wien, Austria}
 \email{armin.rainer@univie.ac.at}

\begin{document}

\begin{abstract}
	We prove a version of Whitney's extension theorem in the ultradifferentiable Beurling setting 
	with controlled loss of regularity.  
	As a by-product we show the existence of continuous linear extension operators on certain 
	spaces of Whitney ultrajets on arbitrary closed sets in $\R^n$.  
\end{abstract}

\thanks{Supported by the Austrian Science Fund (FWF) Project P 32905-N}
\keywords{Whitney extension theorem in the ultradifferentiable setting, Beurling type classes, controlled loss of regularity, 
properties of weight functions}
\subjclass[2020]{
26E10, 
30D60, 
46E10, 
58C25  
}
\date{\today}

\maketitle

\section{Introduction}

Ultradifferentiable versions of Whitney's extension theorem \cite{Whitney34a} 
seek a precise determination of how growth rates of the Whitney jets 
on closed subsets of $\R^n$ are preserved by their extensions to $\R^n$. 
There are different ways of defining these growth rates. 
In this paper they are measured in terms of a weight function $\om$ in the 
framework of so-called Braun--Meise--Taylor classes. 
These classes of ultradifferentiable functions were introduced by Beurling 
\cite{Beurling61} and Bj\"orck \cite{Bjoerck66}. 
We shall work with the reformulation (and generalization) due to 
Braun, Meise, and Taylor \cite{BMT90}.
With each weight function $\om$ two types of Braun--Meise--Taylor classes 
of functions on $\R^n$ can be 
associated: the Beurling type $\cE^{(\om)}(\R^n)$ and the Roumieu type 
$\cE^{\{\om\}}(\R^n)$. 
Similarly, one has for any closed subset $A \subseteq \R^n$ 
the classes $\cE^{(\om)}(A)$ and $\cE^{\{\om\}}(A)$ of Whitney ultrajets on $A$. 
See \Cref{sec:spaces} for precise definitions.  

The jet mapping $j^\infty_A$ which sends a smooth function $f$ on $\R^n$ to the infinite jet $(f^{(\al)}|_A)_{\al}$
consisting of its partial derivatives of all orders restricted to $A$ induces the 
mappings $j^\infty_A : \cE^{(\om)}(\R^n) \to \cE^{(\om)}(A)$ and 
$j^\infty_A : \cE^{\{\om\}}(\R^n) \to \cE^{\{\om\}}(A)$ (by restriction).  
It is natural to ask under which conditions these mappings are surjective.
The answer was given by
Bonet, Braun, Meise, and Taylor \cite{BBMT91} (see also Abanin\cite{Abanin01}) who 
fully characterized those weight functions $\om$ which admit an extension theorem preserving the class: 
The following conditions are equivalent.
\begin{itemize}
	\item $j^\infty_A : \cE^{(\om)}(\R^n) \to \cE^{(\om)}(A)$ is surjective 
	for every closed $A \subseteq \R^n$.
	\item $j^\infty_A : \cE^{\{\om\}}(\R^n) \to \cE^{\{\om\}}(A)$ is surjective 
	for every closed $A \subseteq \R^n$.
	\item $\om$ is a \emph{strong} weight function, i.e.,
	 \begin{equation*} \label{strong}
	 	\exists C>0 ~ \forall t >0 : 
	 	\int_{1}^\infty \frac{\om(tu)}{u^2}\,du \le C\om(t) + C.
	 \end{equation*}
\end{itemize}
Many partial contributions by several authors led to this final answer.

For weight functions $\om$ that are not strong one is interested 
in characterizing weight functions $\si$ 
such that 
\[
	j^\infty_A(\cE^{(\om)}(\R^n)) \supseteq \cE^{(\si)}(A) 
	\quad \text{ and }  \quad
	j^\infty_A(\cE^{\{\om\}}(\R^n)) \supseteq \cE^{\{\si\}}(A),	
\]
since 
the extension involves an unavoidable loss of regularity. 
This question was initiated by Ehrenpreis \cite{Ehrenpreis70} and 
solved for the singleton $A=\{0\}$ by Bonet, Meise, and Taylor 
\cite{BonetMeiseTaylor92} and for 
compact convex sets $A$ with non-empty interior by Langenbruch \cite{Langenbruch94}. 
In our recent papers \cite{Rainer:2019ac,Rainer:2020aa} we solved this problem in 
the Roumieu case for all compact sets $A$; see \Cref{Roumieu} below. 
The purpose of this note is to prove a similar result in the Beurling case.

We want to add that the analogous problems for ultradifferentiable classes 
defined by weight sequences have been solved by Chaumat and Chollet 
\cite{ChaumatChollet94} 
for general $A$, by Langenbruch \cite{Langenbruch94}
for compact convex sets $A$ with non-empty interior, and 
by Petzsche \cite{Petzsche88} and Schmets and Valdivia \cite{SchmetsValdivia04} 
for the singleton.

\subsection{Results} Let us first recall 

\begin{theorem}[{\cite{Rainer:2019ac,Rainer:2020aa}}] \label{Roumieu}
Let $\om$ be a non-quasianalytic concave weight function. 
Let $\si$ be a weight function satisfying $\si(t) = o(t)$ as $t \to \infty$. 
Then we have $j^\infty_A(\cE^{\{\om\}}(\R^n)) \supseteq \cE^{\{\si\}}(A)$
for every closed $A \subseteq \R^n$ if and only if 
\begin{equation}
	\label{mixedstrong} \tag{S}
	 	\exists C>0 ~ \forall t >0 : 
	 	\int_{1}^\infty \frac{\om(tu)}{u^2}\,du \le C\si(t) + C.
\end{equation}
\end{theorem}

Our goal is to prove a version of \Cref{Roumieu} in the Beurling case. 
We follow the standard strategy of reducing the Beurling to the Roumieu case. 
This means roughly speaking that for a pair $(\om,\si)$ of suitable weight functions  
and for a Whitney ultrajet $F$ of class $\cE^{(\si)}$ one tries to find a related pair 
$(\tilde \om,\tilde \si)$ of weight functions which satisfy the assumptions of 
\Cref{Roumieu} and such that the jet $F$ is also of Roumieu class $\cE^{\{\tilde \si\}}$. 
Then by \Cref{Roumieu} there exists an $\cE^{\{\tilde \om\}}$-extension to $\R^n$ which also is of class 
$\cE^{(\om)}$ provided that $\om(t) = o(\tilde \om(t))$ as $t \to \infty$.

We could not preserve the natural condition \eqref{mixedstrong} in the reduction procedure. 
Instead we prove an extension theorem of Beurling type under a slightly stronger condition.   
The following is our main result:

\begin{theorem} \label{Beurling}
Let $\om$ be a non-quasianalytic concave weight function. 
Let $\si$ be a weight function satisfying $\si(t) = o(t)$ as $t \to \infty$. 
Suppose that there exists $r \in (0,1)$ such that 
\begin{equation}
	\tag{S$_{r}$}
	\label{integralcondition}
	~\exists  C>0 ~\forall t >0 : 
	\int_{1}^\infty \frac{\om(tu)}{u^{1+r}}\,du \le C\si(t) + C.	
\end{equation}
Then for every closed $A \subseteq \R^n$ 
we have $j^\infty_A(\cE^{(\om)}(\R^n)) \supseteq \cE^{(\si)}(A)$. 
\end{theorem}

We will see in \Cref{prop:integralcondition} that the existence of $r \in (0,1)$ such that \eqref{integralcondition} holds is equivalent to a 
condition (namely \eqref{strongercondition} below) which is suited for the reduction (see \Cref{lemma2}) 
in the sense that it can be transfered from $(\om,\si)$ to $(\tilde \om,\tilde \si)$. 

\begin{definition}
	We will say that an (ordered) pair of weight functions $(\om,\si)$ is \emph{$r$-strong} if
	\eqref{integralcondition} holds. 
	The pair $(\om,\si)$ is simply called \emph{strong} if it is $1$-strong. Note that 
	the pair $(\om,\om)$ is strong if and only if $\om$ is a strong weight function.
\end{definition}

Note that the condition \eqref{integralcondition} appears naturally in the framework of ultraholomorphic extension theorems
(see \cite{Jimenez-Garrido:2020aa} and \Cref{rem:ultraholomorphic}).

If the weight functions $\om$ and $\si$ are equivalent (i.e., 
generate the same function spaces), 
then the pair $(\om,\si)$ is strong if and only if it is $r$-strong for some $r \in (0,1)$ 
and, additionally, the weight function $\om$ is strong (see \Cref{lemma3}). 
In that case we recover the result stated in the introduction, since any strong weight function is 
equivalent to a concave one and always satisfies $\om(t) = o(t)$ as $t \to \infty$.
In general the condition \eqref{integralcondition} for some $r \in (0,1)$ is strictly stronger than 
\eqref{mixedstrong} (see \Cref{example}). 
Moreover, for the singleton $K = \{0\}$, by \cite{BonetMeiseTaylor92}, 
and for compact convex sets 
$K \subseteq \R^n$ with non-empty interior, by \cite{Langenbruch94}, 
the pair $(\om,\si)$ being strong is equivalent to the inclusion  
$j^\infty_K(\cE^{(\om)}(\R^n)) \supseteq \cE^{(\si)}(K)$. 
So it remains an open question if
in \Cref{Beurling} the condition 
\eqref{integralcondition} can be replaced by \eqref{mixedstrong}.

\begin{remark} \label{remark}
  (1)
  In \Cref{Roumieu} and \Cref{Beurling} we may assume that the ultradifferentiable 
  extension to $\R^n$ of any ultrajet $F$ on $A$ is real analytic 
  on $\R^n \setminus A$.  
  This follows from a result of Schmets and Validivia \cite{Schmets:1999aa}. It can also be seen easily by 
  adapting the proof of Langenbruch \cite[Theorem 13]{Langenbruch03} which is based on a  general 
  approximation theorem of Whitney type developed in the same paper. 

  (2) 
  Both \Cref{Roumieu}(sufficiency of \eqref{mixedstrong}) and \Cref{Beurling} follow from the 
  respective results for compact sets $K \subseteq \R^n$, since both $\cE^{\{\om\}}$ and 
  $\cE^{(\om)}$ admit partitions of unity if $\om$ is non-quasianalytic. 
  For the sake of completeness we give a short argument:
  Let $A \subseteq \R^n$ be closed and fix $F=(F^\al) \in \cE^{[\om]}(A)$; here $\cE^{[\om]}$
  stands for $\cE^{\{\om\}}$ or $\cE^{(\om)}$. 
  For $k \in \N_{\ge 1}$ consider the open annuli $U_k := \{x \in \R^n :  k-1 < |x| < k+1 \}$ and set 
  $U_0 := \{x \in \R^n : |x|<1\}$. 
  There exist functions $\vh_k \in \cE^{[\om]}(\R^n)$ such that $0 \le \vh_k \le 1$, $\on{supp} \vh_k \subseteq U_k$, and 
  $\sum_{k=1}^\infty \vh_k =1$. 
  For each $k \in \N_{\ge 1}$ the jet $F_k := (F^\al|_{\ol U_k})$ belongs to 
  $\cE^{[\om]}(A\cap \ol U_k)$.
  So there exists $f_k \in \cE^{[\om]}(\R^n)$ such that 
  $j^\infty_{A\cap \ol U_k}(f_k) = F_k$. 
  Then
  $f := \sum_{k=1}^\infty \vh_k f_k$
  is a function in $\cE^{[\om]}(\R^n)$ since on any compact set the sum is finite. 
  Let $x \in A$ be fixed. Let $\ell$ be the unique integer such that $\ell-1 \le |x| < \ell$.
  Then $\vh_k =0$ near $x$ unless $k \in \{\ell,\ell+1\}$. 
  For each $\al \in \N^n$  
  \begin{align*}
    f^{(\al)}(x) &= \p^\al \sum_{k=1}^\infty \vh_k(x) f_k(x) 
    \\&= \sum_{\be \le \al} \binom{\al}{\be}\vh^{(\be)}_\ell(x) f^{(\al-\be)}_\ell(x) 
    + \sum_{\be \le \al} \binom{\al}{\be}\vh^{(\be)}_{\ell+1}(x) f^{(\al-\be)}_{\ell+1}(x)
    \\&= \sum_{\be \le \al} \binom{\al}{\be}\p^\be(\vh_\ell(x) + \vh_{\ell+1}(x)) 
     F^{\al-\be}(x) 
    \\&=  F^\al(x), 
  \end{align*}
  since all summands with $|\be|>0$ vanish. 

  (3) In the case that $K \subseteq \R^n$ is a compact set, for any Whitney 
  ultrajet on $K$ the ultradifferentiable extension can be assumed to have compact support 
  by the existence of suitable cut-off functions. In particular, the ultradifferentiable growth
  estimates are global. 
  In \cite{RainerSchindl16a,Rainer:2019ac,Rainer:2020aa} we did justice to this circumstance by 
  writing $\cB^{[\om]}(K)$ for the space of Whitney ultrajets on $K$ and 
  $\cB^{[\om]}(\R^n)$   
  for the space of ultradifferentiable functions on $\R^n$ of \emph{global} class $\cB^{[\om]}$ (cf.\ \Cref{sec:spaces}). 
  Clearly, $\cB^{[\om]}(K) = \cE^{[\om]}(K)$ and $\cB^{[\om]}(\R^n) \subsetneq \cE^{[\om]}(\R^n)$. 

  (4) In \Cref{Roumieu}, \Cref{Beurling}, and also in \Cref{operator} below 
  one could remove the assumption that $\om$ is non-quasianalytic, since the conditions 
  \eqref{mixedstrong}, \eqref{integralcondition} for some $r \in (0,1)$, 
  and $j^\infty_A(\cE^{[\om]}(\R^n)) \supseteq \cE^{[\si]}(A)$ all imply 
  non-quasianalyticity of $\om$. This is obvious for \eqref{mixedstrong} and hence also for \eqref{integralcondition} with $r \in (0,1)$, by 
  \Cref{lem:stronger}. For the third condition 
  it follows from \cite{RainerSchindl15}.   
\end{remark}

As a by-product of the proofs of the \Cref{Roumieu,Beurling} we show in \Cref{operator} 
that, 
for $r$-strong pairs $(\om,\si)$, for $r \in (0,1)$, of weight functions,
there exists a continuous linear extension operator on suitable subspaces of $\cE^{(\si)}(A)$ with values in $\cE^{(\om)}(\R^n)$,
for every non-empty closed 
subset $A \subseteq \R^n$. 
By a \emph{continuous linear extension operator} on a closed set $A \subseteq \R^n$ 
we mean any continuous linear right-inverse of $j^\infty_A$ with suitable domain and codomain. 

\begin{theorem} \label{operator}
Let $\om$ be a non-quasianalytic concave weight function.
Let $\si$ be a weight function satisfying $\si(t)=o(t)$ as $t\to \infty$.
Assume that $(\om,\si)$ is $r$-strong for some $r \in (0,1)$.
If $\ta$ is a weight function such that $\si(t) = o(\ta(t))$ as $t\to \infty$, then for each non-empty closed subset $A \subseteq \R^n$ 
there exists a continuous linear 
extension operator $T_A : \cE^{\{\ta\}}(A) \to \cE^{(\om)}(\R^n)$. 
\end{theorem}

A similar result has been obtained by Chaumat and Chollet \cite[Theorem 31]{ChaumatChollet94} in the setting of Denjoy--Carleman classes. 

In the case that $\om$ is a strong weight function it is well-understood when a continuous linear extension operator 
$\cE^{(\om)}(A) \to \cE^{(\om)}(\R^n)$ exists. 
If $A$ is the closure of a bounded domain with real analytic boundary, it always exists, 
and for the singleton $A=\{0\}$ it exists if and only if 
\[
    \forall C>0 ~\exists \de>0 ~\exists R_0\ge 1 ~\forall R\ge R_0: \om^{-1}(CR) \om^{-1}(\de R) \le \om^{-1}(R)^2;
\]
by Meise and Taylor \cite{MR987760}.
Franken \cite{MR1251459} proved that under this additional condition every closed set $A\subseteq \R^n$ has a 
continuous linear extension operator 
$\cE^{(\om)}(A) \to \cE^{(\om)}(\R^n)$. In the Roumieu case extension operators $\cE^{\{\om\}}(A) \to \cE^{\{\om\}}(\R^n)$ do usually not exist; 
see Langenbruch \cite{MR964961}.

The proof of \Cref{Beurling} is given in \Cref{sec:proofBeurling}; it is based on the Reduction lemma~\ref{lemma2} 
presented in \Cref{sec:reduction}.  \Cref{operator} is proved in \Cref{sec:operator}.

\subsection*{Acknowledgement} 
The author wishes to thank David Nenning and Gerhard Schindl for helpful discussions and reading a 
preliminary version of the paper.

\section{Spaces of ultradifferentiable functions and jets} \label{sec:spaces}

\subsection{Weight functions} \label{sec:weights}
A \emph{weight function} is any continuous increasing function $\om : [0,\infty) \to [0,\infty)$ with $\om(0) =0$ that satisfies
\begin{align}
   & \om(2t) = O(\om(t)) \quad\text{ as } t \to \infty, \label{defom1}\\
   & \log t = o(\om(t)) \quad\text{ as } t \to \infty, \label{defom3}\\
   & \vh(t) := \om(e^t) \text{ is convex}.  \label{defom4}
\end{align} 
A weight function is called \emph{non-quasianalytic} if
\begin{equation}
   \int_1^\infty \frac{\om(t)}{t^2} \, dt <\infty.
 \end{equation}
Two weight functions $\om$ and $\si$ are said to be \emph{equivalent} if $\om(t) = O(\si(t))$ and $\si(t) = O(\om(t))$ 
as $t \to \infty$. 
For each weight function $\om$ there is an equivalent weight function 
$\tilde \om$ such that $\om(t) = \tilde \om(t)$ for large $t>0$ 
and $\tilde \om |_{[0,1]} =0$. It is thus no restriction to assume that $\om |_{[0,1]} =0$ when necessary.

The \emph{Young conjugate} $\vh^*$ of $\vh$ is defined by  
\[
  \vh^*(t) := \sup_{s\ge 0} \big(st-\vh(s)\big), \quad t \ge 0.
\]
Assuming $\om |_{[0,1]} =0$, we have that    
$\vh^*$ is a convex increasing function satisfying $\vh^*(0)=0$, $t/\vh^*(t) \to 0$ as $t \to \infty$, and $\vh^{**}=\vh$; 
cf.\ \cite{BMT90} and \cite[Remark 1.2]{BBMT91}.

\subsection{Function spaces}
Let $\om$ be a weight function, $U \subseteq \R^n$ open, and $\rh >0$.  
We consider the Banach space 
$\cE^{\om}_\rh(U) := \{f \in C^\infty (\R^n) : \|f\|^\om_{U,\rh}< \infty\}$, 
where  
\[
  \|f\|^\om_{U,\rh} := \sup_{x \in U,\,\al \in \N^n} |\p^\al f(x)| \exp(-\tfrac{1}{\rh} \vh^*(\rh |\al|)),
\]
and the locally convex spaces
\begin{align*}
	\cE^{(\om)}(\R^n) 
	&:= \on{proj}_{U \Subset \R^n} \on{proj}_{k \in \N_{\ge 1}} \cE^{\om}_{1/k}(U),
	\\
	\cE^{\{\om\}}(\R^n) 
	&:= \on{proj}_{U \Subset \R^n} \on{ind}_{k \in \N_{\ge 1}} \cE^{\om}_k(U).
\end{align*}
Then $\cE^{(\om)}(\R^n)$ and $\cE^{\{\om\}}(\R^n)$ are called 
\emph{Braun--Meise--Taylor classes} 
of \emph{Beurling type} and of \emph{Roumieu type}, respectively.

Let $\om$ and $\si$ be weight functions. We have the following inclusion relations:
\begin{align*}
	\cE^{(\om)} (\R^n) \subseteq \cE^{(\si)} (\R^n) 
	\quad&\Leftrightarrow\quad
	\si(t) = O(\om(t)) \text{ as } t \to \infty,
	\\
	\cE^{\{\om\}} (\R^n) \subseteq \cE^{\{\si\}} (\R^n) 
	\quad&\Leftrightarrow\quad
	\si(t) = O(\om(t)) \text{ as } t \to \infty,
	\\
	\cE^{\{\om\}} (\R^n) \subseteq \cE^{(\si)} (\R^n) 
	\quad&\Leftrightarrow\quad
	\si(t) = o(\om(t)) \text{ as } t \to \infty,
\end{align*}
cf.\ \cite[Corollary 5.17]{RainerSchindl12}; in particular, $\om$ and $\si$ are equivalent if and only if 
$\cE^{(\om)}(\R^n) = \cE^{(\si)}(\R^n)$, respectively 
$\cE^{\{\om\}}(\R^n) = \cE^{\{\si\}}(\R^n)$.  
The spaces $\cE^{(\om)}(\R^n) \subseteq \cE^{\{\om\}}(\R^n)$ contain non-trivial functions with compact support if and only if $\om$ is 
non-quasianalytic (cf.\ \cite{BMT90} or \cite{RainerSchindl12}). 
Since $C^\om(\R^n) = \cE^{\{t\}}(\R^n)$, the space of real analytic function on $\R^n$ is contained in $\cE^{\{\om\}}(\R^n)$ 
if and only if $\om(t) = O(t)$ as $t \to \infty$, and it is contained in $\cE^{(\om)}(\R^n)$ if and only if 
$\om(t) = o(t)$ as $t \to \infty$.

\subsection{Whitney ultrajets}
Let $A\subseteq \R^n$ be a closed non-empty set. Let $\cE(A)$ be the set of \emph{Whitney jets} 
(of class $C^\infty$) on $A$, i.e.,  
$F=(F^\al)_{\al \in \N^n} \in  C^0(A,\R)^{\N^n}$ belongs to  
$\cE(A)$ if for all compact subsets $K \subseteq A$, all $p \in \N$, and all $|\al|\le p$ we have  
\[ 
(R^p_xF)^\al (y) = o(|y-x|^{p-|\al|}) \quad\text{ as } |y-x| \to 0, \, x,y \in K, 
\]
where 
 \[
 	(R^p_xF)^\al (y):= F^\al(y) - \sum_{|\be|\le p -|\al|} \frac{(y-x)^\be}{\be!} F^{\al+\be}(x).  
 \]
By Whitney's extension theorem \cite{Whitney34a}, the mapping $j^\infty_A : C^\infty(\R^n) \to \cE(A)$, 
which is well-defined by Taylor's theorem, is surjective.   

Let $\om$ be a weight function.  
A Whitney jet $F=(F^\al)_{\al \in \N^n} \in \cE(A)$ 
is called a \emph{$\om$-Whitney ultrajet of Beurling type} on $A$ 
if for all compact subsets $K \subseteq A$ and all integers $m \ge 1$ 
we have
\begin{equation} \label{jet1}
 	\|F\|^\om_{K,1/m}:=\sup_{x \in K} \sup_{\al \in \N^n} |F^\al(x)| \exp\big(-m \vh^*\big(\tfrac{|\al|}{m}\big)\big) < \infty
 \end{equation} 
 and 
 \begin{equation} \label{jet2}
 	| F |^\om_{K,1/m}  := \sup_{\substack{x,y \in K\\ x \ne y}} \sup_{p \in \N } \sup_{|\al| \le p}  |(R^p_x F)^\al (y)|  
 	\frac{(p+1-|\al|)!}{|x-y|^{p+1-|\al|}} \exp\big(-m \vh^*\big(\tfrac{p+1}{m}\big)\big) <\infty.   
 \end{equation}
 We denote by $\cE^{(\om)}(A)$ 
 the locally convex space of all $\om$-Whitney ultrajets $F$ of Beurling type on $A$ equipped with the project limit topology 
 with respect to the system of seminorms $\|F\|^\om_{K,1/m} + |F|^\om_{K,1/m}$. 

The space of \emph{$\om$-Whitney ultrajets of Roumieu type} on $A$ 
 is
 \[
 \cE^{\{\om\}}(A) := \{F \in \cE(A) : \forall K \Subset A ~\exists m\in \N_{\ge 1} : \|F\|^\om_{K,m} + |F|^\om_{K,m}<\infty \}
 \]
supplied with its natural locally convex topology.

\section{The condition \texorpdfstring{\eqref{integralcondition}}{(S$_r$)}} 
In this section we discuss $r$-strong pairs of weight functions.

\begin{lemma} \label{lem:stronger}
Let $\om$ and $\si$ be weight functions.
If the pair $(\om,\si)$ is $r$-strong for some $r \in (0,1)$, then it is $s$-strong for all $s \in [r,1]$.
In particular, the pair $(\om,\si)$ is strong and we have $\om(t) = O(\si(t))$ as $t \to \infty$. 
\end{lemma}

\begin{proof}
	The first part of the lemma is trivial, since
	\begin{equation} \label{trivial}
		\int_1^\infty  \frac{\om(ut)}{u^{1+s}} \, du \le \int_1^\infty  \frac{\om(ut)}{u^{1+r}} \, du	
	\end{equation}
	 for all 
	$s \in [r,1]$.  
  Since $\om$ is increasing, we have 
  $\om(t) = \om(t) \int_1^{\infty} \frac{du}{u^2} \le \int_1^{\infty} \frac{\om(ut)}{u^2} \, du$ 
  which implies the supplement.
\end{proof}

In the following proposition we discover a condition which is equivalent of the pair $(\om,\si)$ being $r$-strong for some $r \in (0,1)$. 
This condition can be preserved in the reduction procedure (see \Cref{lemma2}). 

\begin{proposition} \label{prop:integralcondition}
	Let $\om$ and $\si$ be weight functions. Then the pair $(\om,\si)$ is $r$-strong for some $r \in (0,1)$
	if and only if 
		\begin{equation}
		\label{strongercondition}
   		 \exists C>0 ~\exists K>H> 1   ~\exists t_0 \ge 0 ~\forall t\ge t_0 
    	~\forall j\in \N_{\ge1} :
    	\om(K^jt) \le C H^j \si(t).
	\end{equation}
\end{proposition}

\begin{proof}
	If $\om$ and $\si$ satisfy \eqref{integralcondition} with $r \in (0,1)$ and 
	$K>1$ is arbitrary, then since $\om$ is increasing, for all integers $j\ge0$, we have   
	\begin{align} \label{computation1}
		\frac{\om(K^j t)}{K^{rj}} 
		&\le \sum_{i=0}^\infty \frac{\om(K^i t)}{K^{ri}} 
		= \frac{K^{1+r}}{K-1} \sum_{i=0}^\infty \frac{\om(K^i t)}{K^{(1+r)(i+1)}} (K^{i+1}-K^i)
		\\
		&\le \frac{K^{1+r}}{K-1} \int_1^\infty \frac{\om(tu)}{u^{1+r}} \, du
		\le C \frac{K^{1+r}}{K-1} ( \si(t) +1 ). \nonumber 
	\end{align}
	This implies \eqref{strongercondition} with $H:=K^r$.

	Suppose that \eqref{strongercondition} is fulfilled. 
	Then we have $H = K^{r_0}$ with $r_0 := \log H/\log K \in (0,1)$. 
	Let $r \in (r_0,1)$. Then for sufficiently large $t$, 
	\begin{align*}
		\int_1^\infty  \frac{\om(tu)}{u^{1+r}} \, du 
		&=
		\sum_{j=0}^\infty \int_{K^j}^{K^{j+1}} \frac{\om(ut)}{u^{1+r}} \, du 
		\le  
		\sum_{j=0}^\infty \om(K^{j+1} t) \int_{K^j}^{K^{j+1}} \frac{1}{u^{1+r}} \, du
		\\
		&\le   
		\sum_{j=0}^\infty C  H^{j+1}\si(t) r^{-1} (K^{-rj} - K^{-r(j+1)}) 
		\\
		&= C r^{-1} (K^r-1) \si(t) \sum_{j=0}^\infty (K^{r_0-r})^{j+1}
		= C_1 \si(t),
	\end{align*} 
	since $K^{r_0-r}<1$. Thus $(\om,\si)$ is $r$-strong. 
\end{proof}

The following lemma shows that \Cref{Beurling} generalizes the extension result 
stated in the introduction.

\begin{lemma} \label{lemma3}
  Let $\om$ be a weight function. The following conditions are equivalent:
  \begin{enumerate}[label=\textnormal{($\arabic*$)}]
  		\item  $\om$ is strong.
  		\item  Condition \eqref{strongercondition} holds with $\si=\om$ and $C=1$.
  		\item  $(\om,\om)$ is $r$-strong for some $r \in (0,1)$.
  	\end{enumerate}	
\end{lemma}

\begin{proof}
  (1) $\Rightarrow$ (2):
  It was shown in \cite{MeiseTaylor88a} that 
  the $\om$ being strong is equivalent to 
  \[
    \exists K>1 : \limsup_{t \to \infty} \frac{\om(Kt)}{\om(t)} <K.
  \]
  That means there exist $\ep>0$ and $t_0 \ge 0$ such that
  \begin{equation*} 
    \frac{\om(Kt)}{\om(t)} \le K-\ep, \quad \text{ for all } t \ge t_0.  
  \end{equation*}
  Thus condition \eqref{strongercondition} with $C=1$, $H = K-\ep$, and $\si=\om$ follows by iteration. 
 
  (2) $\Rightarrow$ (3) follows from \Cref{prop:integralcondition}.

  (3) $\Rightarrow$ (1) is a consequence of	\Cref{lem:stronger}.   
\end{proof}

Let us discuss the relevance of condition \eqref{strongercondition} (or equivalently \eqref{integralcondition} for some $r\in (0,1)$) 
for the extension 
problem. 
Suppose that $\om$ is a weight function that is not strong and let $A \subseteq \R^n$ be a closed non-empty set.
We consider $j^\infty_A \cE^{(\om)}(\R^n)$ and look for spaces of 
Whitney ultrajets $\cE^{(\si)}(A)$ as large as possible and contained in $j^\infty_A \cE^{(\om)}(\R^n)$. 
If $\ol \om$ is any strong weight function such that 
$\om(t) = O(\ol \om(t))$ as $t \to \infty$, 
then $\cE^{(\ol \om)}(A) \subseteq  j^\infty_A \cE^{(\om)}(\R^n)$, since 
$\cE^{(\ol \om)}(A) = j^\infty_A \cE^{(\ol \om)}(\R^n)$ and 
$\cE^{(\ol \om)}(\R^n) \subseteq \cE^{(\om)}(\R^n)$. 
The following \Cref{relevance} and \Cref{Beurling} show 
that for any such $\ol \om$ there is a weight function $\si$ 
such that  $\cE^{(\ol \om)}(A) \subsetneq \cE^{(\si)}(A) \subseteq  j^\infty_A \cE^{(\om)}(\R^n)$ 
(independently of $A$).

\begin{proposition} \label{relevance}
  Let $\om$ be a weight function that is not strong.
  For any strong weight function $\ol \om$ with 
  $\om(t) = O(\ol \om(t))$ as $t \to \infty$ 
  there exists a weight function $\si$ such that 
  the pair $(\om,\si)$ is $r$-strong for some $r \in (0,1)$
  and we have 
  $\si(t) = O(\ol \om(t))$, $\ol \om(t) \ne O(\si(t))$,  
  and $\si(t) =o(t)$ as $t \to \infty$.    
\end{proposition}

\begin{proof}
  By \Cref{lemma3}  
  there exist constants $r_0 \in (0,1)$, $C>0$, and $t_0\ge 0$ such that 
  \begin{equation} \label{majorant}
  	\int_1^\infty \frac{\ol \om(tu)}{u^{1+r_0}} \, du \le C\, \ol \om(t) \quad \text{ for all } t \ge t_0.
  \end{equation}
  Since $\om(t) = O(\ol \om(t))$ as $t \to \infty$,
  there is a constant $C_1>0$ and $t_1 \ge t_0$ such that 
  $\om(t) \le C_1\, \ol \om(t)$ for all $t\ge t_1$.

  Let $r \in (r_0,1)$ and consider 
  \[
  	\si_r(t) := \int_1^\infty \frac{\om(tu)}{u^{1+r}} \, du, \quad t \ge 0. 
  \]
  The integral converges and we have $\si_r (t) \le C_2 \ol \om(t)$ for all $t \ge t_1$, by \eqref{majorant}. 
  Moreover, $\si_r(t) \le \si_s(t)$ for all $t \ge t_1$ provided that $r_0 <s < r<1$. Since $\om$ is increasing 
  we have $\om(t) \le r \si_r(t) \le \si_r(t)$ for all $r$ and all $t$. 

  Clearly, the pair $(\om,\si_r)$ is $r$-strong 
  for all $r \in (r_0,1)$. 
  By the inequalities $\om(t) \le \si_r(t) \le C_2 \ol \om(t)$ for all 
  $t\ge t_1$ we may conclude that $\si_r(t) =o(t)$ as $t \to \infty$, since every strong weight function
  satisfies this, and $\log(t) = o(\si_r(t))$ as $t \to \infty$, since $\om$ has this property (cf.\ \eqref{defom3}).
  There is a constant $C_3>0$ such that $\si_r(2t) \le C_3\, \si_r(t)$ for large $t$ since 
  $\om$ has this property  (cf.\ \eqref{defom1}).
  Since $t \mapsto \om(e^t)$ is convex  (cf.\ \eqref{defom4}), 
  it follows that also $t \mapsto \si_r(e^t)$ is convex for all $r$. 
  Consequently, $\si_r$ is continuous and it is increasing since $\om$ is increasing.
  So $\si_r$ is a weight function in the sense of \Cref{sec:weights}.

  The computation \eqref{computation1} shows that for all $r \in (0,1)$ and all $t\ge0$, 
  \begin{equation} \label{interlace1}
  	\ta_r(t) := \sup_{j \in \N} \frac{\om(K^j t)}{K^{rj}}  \le C_4\, \si_r(t).
  \end{equation}
  On the other hand, for $0<s<r$,
  \begin{align} \label{interlace2}
  	\si_r(t) &= \int_1^\infty \frac{\om(ut)}{u^{1+r}}\, du  = \sum_{j=0}^\infty \int_{K^j}^{K^{j+1}} \frac{\om(ut)}{u^{1+r}}\, du
  	\\
  	&\le  \sum_{j=0}^\infty  \frac{\om(K^{j+1}t)}{K^{j(1+r)}} (K^{j+1}-K^j)
  	= K^r(K-1)\sum_{j=0}^\infty  \frac{\om(K^{j+1}t)}{K^{s(j+1)}} \frac{K^{s(j+1)}}{K^{r(j+1)}} \nonumber
  	\\
  	&\le K^r(K-1) \sum_{j=0}^\infty (K^{s-r})^{j+1} \cdot  \ta_s(t) = C_5 \, \ta_s(t) \nonumber
  \end{align}
  for all $t \ge 0$. Arguing similarly as for $\si_r$ one sees easily that $\ta_r$, for $r \in (r_0,1)$, form weight functions 
  satisfying $\ta_r(t) = O(\ol \om(t))$ and $\ta_r(t) =o(t)$ as $t \to \infty$. Moreover,  $\ta_r$ satisfies \eqref{strongercondition} relative to 
  $\om$ by definition. We will use the interlacing properties \eqref{interlace1} and \eqref{interlace2} in order to show that  
  $\ol \om(t) \ne O(\si_r(t))$ as $t \to \infty$ for all $r \in (r_0,1)$. 
  In fact, we will show that 
  \begin{equation} \label{relevance:claim}
  	\ta_s(t) \ne O(\ta_r(t)), \quad \text{ if }  s<r,	
  \end{equation}
  which implies $\si_s(t) \ne O(\si_r(t))$ as $t \to \infty$, 
  by \eqref{interlace1} and \eqref{interlace2}. Consequently, for given $r \in (r_0,1)$ take $s \in (r_0,r)$ and we see that 
  $\si_{s} (t) = O(\ol \om(t))$ and $\si_s(t) \ne O(\si_{r}(t))$ as $t \to \infty$ imply 
  that $\ol \om(t) \ne O(\si_r(t))$ as $t \to \infty$.

  Let us prove \eqref{relevance:claim}. Two cases may occur: 
  First, if there exists $N \in \N$ and $T > 0$ such that for all $t \ge T$
  \[
 	\sup_{j \in \N} \frac{\om (K^j t)}{K^{sj}} 
 	= \max_{0 \le j \le N} \frac{\om (K^j t)}{K^{sj}}, 
  \]
  then $\ta_s(t) \le \om(K^N t)$ for such $t$ and since $\om(2t) = O(\om(t))$ as $t \to \infty$ (cf.\ \eqref{defom1}),
  we may conclude that the weight functions $\om$ and $\ta_s$ are equivalent. 
  As a consequence $\om$ is a strong weight function (by \Cref{lem:stronger}) contrary to the assumption of the 
  proposition.
  Otherwise there are sequences $j_n\to \infty$ and $t_n \to \infty$ 
  such that 
  \[
  	\frac{\om (K^{j_n} t_n)}{K^{r j_n}} 
  	\ge  \sup_{j \in \N} \frac{\om (K^j t_n)}{K^{rj}} - \frac{1}{n} 
  	= \ta_r(t_n)- \frac{1}{n}.  
  \]
  Consequently,
  \[
  	\ta_s(t_n) \ge \frac{\om (K^{j_n} t_n)}{K^{sj_n}} 
  	= K^{(r-s)j_n} \frac{\om (K^{j_n} t_n)}{K^{rj_n}} 
  	\ge K^{(r-s)j_n} \Big(\ta_r(t_n)- \frac{1}{n}\Big).
  \]
  This implies \eqref{relevance:claim}. The proof is complete.  
\end{proof}

\begin{remark}
  The concave weight function $\kappa(t) := \int_1^\infty \frac{\om(ty)}{y^2}\, dy$ (cf.\ \cite[Remark 3.20]{MeiseTaylor88a} and 
  \cite[Proposition 1.3]{BonetMeiseTaylor92}) 
  satisfies 
  \[
    \kappa(t) \ge \frac{\om(K^jt)}{K^j} \quad \text{ for all } t>0,~K>1,~j \in \N. 
  \]
  Indeed, this follows from $\om \le \ka$ and concavity of $\ka$ or more directly by
  \begin{multline*}
    \kappa(t) 
    = \int_1^\infty \frac{\om(ty)}{y^2}\, dy
    = \int_{K^{-j}}^\infty \frac{\om(K^jt u)}{K^ju^2}\, du
    \ge \frac{\om(K^jt)}{K^j} \int_{1}^\infty \frac{1}{u^2}\, du
    = \frac{\om(K^jt)}{K^j},
  \end{multline*}
  using that $\om$ is increasing. 
  Note that $\ka$ defines the largest class of ultradifferentiable functions (respectively, jets) 
  among all weight functions $\si$ such that $\int_1^\infty \frac{\om(ty)}{y^2}\, dy = O(\si(t))$ as $t \to \infty$ (i.e.\ 
  such that the pair $(\om,\si)$ is strong). 
  In the situation of \Cref{relevance} we have 
  $\om(t) \le \ka(t) \le C_1\, \si_r(t) \le C_2\, \ol \om(t)$ for sufficiently large $t$ and $r \in (r_0,1)$ 
  (thanks to \eqref{trivial}). 
\end{remark}

\begin{example} \label{example}
	For $\al>0$ we consider the weight functions $\om_\al$ where 
	\[
		\om_\al(t) = \frac{t}{(\log t)^\al}, \quad \text{ for sufficiently large } t. 
	\]
	It is easy to see that the definition of $\om_\al$ can be extended to small positive values of $t$ such that 
	it is a weight function in the sense of \Cref{sec:weights}. As observed in \cite[Example 3.8]{BonetMeiseTaylor92} 
	if $\al>1$ then
	\[
		\kappa_\al(t) := \int_1^\infty \frac{\om_\al(ty)}{y^2}\, dy = \frac{\om_{\al-1}(t)}{\al-1}, \quad \text{ for large } t.
	\]
	(Apparently there is a small inaccuracy concerning the constant $\al-1$ in the above reference.) 
	This shows that for $\al>1$ the pair $(\om_\al,\om_{\al-1})$ is strong
	whereas the weight function $\om_\al$ is not strong. 

	On the other hand
	the pair $(\om_\al,\om_{\al-1})$ is not $r$-strong for any $r \in (0,1)$. Indeed it is easy to see 
	that the condition \eqref{strongercondition} is violated. 
	Otherwise there would exist constants $K>H>1$ such that for all integers $j \ge 1$ and sufficiently large $t$ 
	\[
		\frac{\om_\al(K^jt)}{H^j\om_{\al-1}(t)} = 	 \Big(\frac{K}{H}\Big)^j \frac{(\log t)^{\al-1}}{(\log t + j \log K)^\al}
	\]
	is bounded which is obviously a contradiction. 
	Note that 
	\[
		\int_1^\infty \frac{\om_\al(ty)}{y^{1+r}}\, dy = \om_{\al-1}(t) \cdot t^{r-1} E_\al((r-1)\log t),  	
	\]
	where $E_\al(x) := \int_1^\infty y^{-\al} e^{-x y}\, dy$ is the exponential integral.
\end{example}

\begin{remark} \label{rem:ultraholomorphic}
	The condition \eqref{integralcondition} plays an interesting role in the context of ultraholomorphic sectorial extensions. 
	The growth index $\ga(\si,\om)$ introduced in \cite{Jimenez-Garrido:2020aa} is defined by 
	\[
		\ga(\si,\om) := \sup \{ s>0 : (\om,\si) \text{ ist $1/s$-strong}\}.	
	\]  
	Note that a pair of weight functions $(\om,\si)$ is $r$-strong for some $r \in (0,1)$ if and only if $\ga(\si,\om)>1$. 
	The growth index for a pair of weight functions is a generalization of the growth index $\ga(\om) := \ga(\om,\om)$ 
	of a single weight function $\om$
	considered in \cite{Jimenez-Garrido:2019aa,Jimenez-Garrido:2019ab} imitating the growth index $\ga(M)$ introduced by 
	\cite{MR2011916} for a weight sequence $M$. 
	In the mentioned papers extension results of Borel--Ritt type are proved: given a formal power series with admissible growth behavior 
	of the coefficients one looks for
	an ultraholomorphic function defined on a sector in the Riemann surface of the logarithm and asymptotic to the 
	given series.
	The growth indices give sharp upper bounds for the aperture of the sectors on which the extension exists.
\end{remark}

\section{The reduction lemma} \label{sec:reduction}

The following lemma is the key to \Cref{Beurling}. Its proof contains ideas from 
\cite[Lemma 1.6, Lemma 1.7]{BMT90} and \cite[Lemma 4.4]{BBMT91}.

\begin{lemma}[Reduction lemma] \label{lemma2}
  Let $\omega$ and $\sigma$ be weight functions satisfying the following conditions:
  \begin{enumerate}[label=\textnormal{($\alph*$)}]
  	\item $\om$ is concave. \label{i:1}
    \item $\si(t) = o(t)$ as $t \to \infty$.  \label{i:2}
    \item $(\om,\si)$ is $r$-strong for some $r \in (0,1)$. \label{i:3}
  \end{enumerate}  
  Suppose that $f : [0,\infty) \to [0,\infty)$ is a function satisfying $\si(t) = o(f(t))$ as $t \to \infty$.
  Then there exist weight functions $\tilde \omega$ and $\tilde \sigma$ 
  such that 
  \[
  	\text{$\om(t) = o(\tilde \om(t))$, $\si(t) = o(\tilde \si(t))$, and 
  	$\tilde \si(t) = o(f(t))$ as $t \to \infty$}
  \]
  and with the following properties:
  \begin{enumerate}[label=\textnormal{($\alph*$)}]
  	\item[\mylabel{t:1}{\thetag{$\tilde a$}}] $\tilde \om$ is concave.
    \item[\mylabel{t:2}{\thetag{$\tilde b$}}]  
    $\tilde \si(t) = o(t)$ as $t \to \infty$. 
    \item[\mylabel{t:3}{\thetag{$\tilde c$}}] $(\tilde \om,\tilde \si)$ is $r$-strong for some $r \in (0,1)$.
  \end{enumerate}
\end{lemma}

\begin{proof}
   We first claim that $\om$ is of class $C^1$.
   Since $\om$ is concave, for each $t>0$ the left-sided and the right-sided derivatives 
   $\om^{\prime-}(t)$ and $\om^{\prime+}(t)$ exist and satisfy 
   $\om^{\prime+}(t) \le \om^{\prime-}(t)$. Since $\vh = \om \o \exp$ is convex, we 
   also have $\om^{\prime-}(t) \le \om^{\prime+}(t)$. 
   It follows that $\om$ is differentiable. 
   The monotonicity properties imply that $\om$ is even of class $C^1$, indeed if 
   $s \le t$ then $\om'(t) \le \om'(s)$ and $\om'(s)s \le \om'(t)t$ and hence
   \[
      \om'(t) \le \om'(s) \le \om'(t)\frac{t}{s} \to \om'(t) \quad \text{ as } t\to s. 
   \] 

   The condition $\om(t) = o(t)$ as $t \to \infty$ implies that $\om'(t) \searrow 0$ as $t \to \infty$; indeed otherwise 
   $\om'(t)\ge \ep>0$ for all $t$ as $\om'$ is decreasing and hence 
   $\om(t) = \int_0^t \om'(s)\, ds \ge \ep t$, a contradiction.

  \subsection*{Auxiliary sequences}
   By \Cref{prop:integralcondition}, assumption \ref{i:3} implies that 	
   $\om$ and $\si$ satisfy \eqref{strongercondition}. So  
   there exist constants $C>0$, $K>H>1$ and $t_0\ge0$ such that 
   \begin{equation} \label{str}
   	\om(K^jt) \le C H^j \si(t), \quad \text{ for all } t\ge t_0, \, j\ge 1.					
   \end{equation}				
   Note that $\log(t) = o(\si(t))$ and $\si(t) = o(f(t))$ imply $f(t) \to \infty$ as $t \to \infty$.  
   We define inductively three sequences $(x_n)$, $(y_n)$, and $(z_n)$ with $x_1=y_1=z_1 =0$, $x_2> 0$, and the following 
   properties:
   \begin{gather}
         \label{px1}
         x_n > \max\{2,K\}  y_{n-1} + n,
         \\
         \label{px4}
         \min\{t, f(t)\} \ge n^2 \si(t), \quad \text{ for all } t\ge x_n, 
         \\
         \label{px3}
         \om(x_n) \ge 2^{n-i} \om(z_i), \quad 1 \le i \le n-1,
         \\
         \label{px5}
         \si(x_n) \ge 2^{n-i} \si(x_i), \quad 1 \le i \le n-1,
         \\
         \label{py1}
         \om'(y_n) = \frac{n-1}{n} \om'(x_n),
         \\
         \label{pz1}
         \om(z_n) = n \om(y_n) - (n-1) \big(\om(x_n) + (y_n-x_n) \om'(x_n)\big).  
   \end{gather}
   The point $y_n$ is well-defined by \eqref{py1} since $\om'(t) \searrow 0$ as $t \to \infty$. 
   The condition \eqref{pz1} is equivalent to 
   \[
      \om(y_n) - \om(z_n) = (n-1) \big(\om(x_n) + (y_n-x_n) \om'(x_n) - \om(y_n) \big),
   \]
   that means the difference $\om(y_n) - \om(z_n)$ is the $(n-1)$-fold of the distance between the point 
   with the abscissa $y_n$ 
   on the tangent line at $x_n$ to the graph of $\om$ and the point on the graph with the same abscissa. 
   That $z_n \in [x_n,y_n]$ with this property exists follows from the assumption that $\om$ is concave. Indeed, 
   by continuity of $\om$ it suffices to check that    
   \[
      \om(y_n) - \om(x_n) \ge  (n-1) \big(\om(x_n) + (y_n-x_n) \om'(x_n) - \om(y_n) \big)
   \]
   which is equivalent to 
   \[
      \frac{\om(y_n) - \om(x_n)}{y_n-x_n} \ge  \frac{n-1}{n}\om'(x_n)  \stackrel{\eqref{py1}}{=} \om'(y_n)
   \]
   and holds by concavity of $\om$.

   \subsection*{Definition of $\tilde \omega$}
   We define $\tilde \omega$ as follows:
   \[
     \tilde \omega(t) := 
      \begin{cases}
        (n-1)  \big(\om(x_n) + (t-x_n) \om'(x_n) \big) - \sum_{i=1}^{n-2} \om(z_{i+1}) & \text{ if } x_n \le t < y_n,
        \\
        n \om(t) - \sum_{i=1}^{n-1} \om(z_{i+1}) & \text{ if } y_n \le t < x_{n+1}.
      \end{cases} 
   \]
   Then $\tilde \omega$ is continuous (note that continuity at $y_n$ follows from \eqref{pz1})
   and continuously differentiable (which follows from \eqref{py1}).  
   Since $\om$ is increasing we may infer that $\tilde \om$ is increasing. 
   Being a $C^1$-function which piecewise is either affine-linear or concave $\tilde \om$ must be concave. 
   The same reasoning also shows that $t \mapsto \tilde w(e^t)$ is convex, since $\om$ has this property. 

   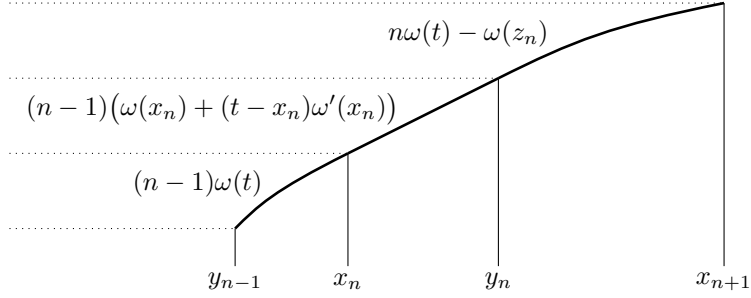
\begin{figure}[ht]
   	
   \begin{tikzpicture}
   	
   	\draw (1.5,0) -- (1.5,-0.5);
   	\draw (3,1) -- (3,-0.5);
   	\draw (5,2) -- (5,-0.5);
   	\draw (8,3) -- (8,-0.5);

   	\draw[dotted] (-1.5,0) -- (1.5,0);
   	\draw[dotted] (-1.5,1) -- (3,1);
   	\draw[dotted] (-1.5,2) -- (5,2);
   	\draw[dotted] (-1.5,3) -- (8,3);

   	\draw[line width=1pt] (3,1) -- (5,2);
   	\draw[line width=1pt] (5,2) .. controls (6,2.5) and (6.5,2.7) .. (8,3);
   	\draw[line width=1pt] (1.5,0) .. controls (1.8,0.3) and (2,0.5) .. (3,1);

   	\node at (1.5,-0.7) {$y_{n-1}$};
   	\node at (3,-0.7) {$x_{n}$};
   	\node at (5,-0.7) {$y_{n}$};
   	\node at (8,-0.7) {$x_{n+1}$};

   	\node at (1,0.6) {$(n-1) \om(t)$};
   	\node at (1.2,1.6) {$(n-1)  \big(\om(x_n) + (t-x_n) \om'(x_n) \big)$};
   	\node at (4.6,2.6) {$n \om(t) - \om(z_{n})$};

   \end{tikzpicture}
   \caption{The graph of $[y_{n-1},x_{n+1}] \ni t \mapsto \tilde \om(t) + \sum_{i=1}^{n-2} \om(z_{i+1})$.}
   \end{figure}

   Next we claim that 
   \begin{align} \label{claim1}
      \tilde \om(t) \ge (n-2) \om(t),   
      \quad \text{ if } t \in [x_n,x_{n+1}) \text{ and } n \ge 2.        
   \end{align}     
   If $x_n \le t \le y_n$ then by concavity of $\om$
   \begin{align*} 
     \tilde \omega(t) 
     &=  (n-1)  \big(\om(x_n) + (t-x_n) \om'(x_n) \big) - \sum_{i=1}^{n-2} \om(z_{i+1})
     \\
     &\ge (n-1) \om(t) - \sum_{i=1}^{n-2} \om(z_{i+1}) 
     \\
     &= \Big((n-1)  - \sum_{i=1}^{n-2} \frac{\om(z_{i+1})}{\om(t)} \Big) \om(t)
     \\
     &\ge (n-2) \om(t), 
   \end{align*}
   since by \eqref{px3}
   \[
      \sum_{i=1}^{n-2} \frac{\om(z_{i+1})}{\om(t)} 
      \le \sum_{i=1}^{n-2} 2^{i+1-n} = 1 - 2^{2-n}.
   \]
   If $y_n \le t < x_{n+1}$ then $\om(t) - \om(z_n) \ge 0$, since $\om$ is increasing and $z_{n} \le y_n$, and 
   thus   
    \begin{align*} 
     \tilde \omega(t) 
     = n \om(t) - \sum_{i=1}^{n-1} \om(z_{i+1}) 
     \ge (n-1) \om(t) - \sum_{i=1}^{n-2} \om(z_{i+1}) 
     \ge (n-2) \om(x)
   \end{align*}
   as in the previous case. Hence \eqref{claim1} is proved.

   Note that \eqref{claim1} shows $\om(t) = o(\tilde \om(t))$  as $t \to \infty$; in particular, $\log(t) = o(\tilde \om(t))$  as $t \to \infty$,
   since $\om$ has this property. 

   On the other hand we have 
   \begin{align} \label{claim2}
      \tilde \om(t) \le n \om(t),   
      \quad \text{ if } t \in [x_n,x_{n+1}) \text{ and } n \ge 2.        
   \end{align}
   This is clear by definition if $y_n \le t < x_{n+1}$. 
   On the interval $[x_n,y_n]$ the graph of $\tilde \om(x)$ is the line segment with endpoints 
   $(x_n,\tilde \om(x_n))$ and $(y_n,\tilde \om(y_n))$. 
   Since $\tilde \om(x_n) = (n-1) \om(x_n) - \sum_{i=1}^{n-2} \om(z_{i+1}) \le (n-1) \om(x_n) \le n \om(x_n)$
   and $\tilde \om(y_n) = n \om(y_n) - \sum_{i=1}^{n-1} \om(z_{i+1}) \le n \om(y_n)$ and 
   $n \om$ is a concave function, we have $\tilde \om(t) \le n \om(t)$ for all $t \in [x_n,y_n]$ as well. 
   Thus \eqref{claim2} is shown.

  \subsection*{Definition of $\tilde \si$}
   We define $\tilde \sigma$ by
   \[
     \tilde \si(t) := n \si(t) - \sum_{i=1}^n \si(x_i), \quad \text{ for  } t \in [x_n,x_{n+1}).  
   \]    
   Then $\tilde \si$ is a continuous increasing functions such that $t \mapsto \tilde \si(e^t)$ is convex.
   We have 
   \begin{align} \label{claim3}
      \tilde \si(t) \ge (n-2) \si(t),   
      \quad \text{ if } t \in [x_n,x_{n+1}) \text{ and } n \ge 2.        
   \end{align}   
   In fact, by \eqref{px5}
  \begin{align*}
       \tilde \si(t) = \Big(n - \sum_{i=1}^n \frac{\si(x_i)}{\si(t)}\Big) \si(t) \ge 
     \Big(n - \sum_{i=1}^n 2^{i-n}\Big) \si(t) \ge (n-2) \si(t).
  \end{align*}
  This shows $\si(t) = o(\tilde \si(t))$  as $t \to \infty$; in particular, 
   $\log(t) = o(\tilde \si(t))$  as $t \to \infty$. 

   By definition we have 
   \begin{align} \label{claim4}
      \tilde \si(t) \le n \si(t),   
      \quad \text{ if } t \in [x_n,x_{n+1}) \text{ and } n \ge 2.        
   \end{align}
   Furthermore,  by \eqref{px4},
  \begin{equation*}
    \tilde \si(t)  
    \le \frac{1}{n} f(t), \quad \text{ for all } t \ge x_n, 
  \end{equation*}
  and hence $\tilde \si(t) = o(f(t))$ as $t \to \infty$, 
  as well as  
  \begin{equation*}
    \tilde \si(t)  
    \le \frac{1}{n} t, \quad \text{ for all } t \ge x_n, 
  \end{equation*}
  which shows $\tilde \si(t) = o(t)$ as $t \to \infty$.

  \subsection*{The pair $(\tilde \om,\tilde \si)$ is $r$-strong for some $r \in (0,1)$}
   Let $N \in \N_{>2}$ be such that $x_N \ge t_0$.
  Let $y \ge x_N$. Suppose that $n \in \N_{\ge N}$ is such that $y \in [x_n,x_{n+1})$. Then  
  $Ky  \in [x_n,x_{n+2})$ since $K x_{n+1} \le x_{n+2}$ by \eqref{px1}.
  By iteration $K^jy \in [x_n,x_{n+j+1}]$.  
  Then  
  \begin{align*}
    \tilde \om(K^j y) &\stackrel{\eqref{claim2}}{\le} (n+j+1) \om(K^j y) 
    \stackrel{\eqref{str}}{\le} (n+j+1)  C H^j \si(y) 
    \\
    &\stackrel{\eqref{claim3}}{\le} \frac{n+j+1}{n-2} C H^j \tilde \si(y)
    \le D j H^j \tilde \si(y)
  \end{align*}
  for all  $y \ge x_N$ and $j \ge 1$ and a constant $D$.   
  Choose $\tilde H \in (H,K)$ and $\tilde N \in \N$ such that 
  \[
    \frac{\log j}{j} \le \log \tilde H - \log H \quad \text{ for all } j \ge \tilde N. 
  \]
  Then $jH^j \le \tilde N \tilde H^j$ for all $j$. 
  It follows that $\tilde \om$ and $\tilde \si$ fulfill \eqref{strongercondition} and thus  
  $(\tilde \om,\tilde \si)$ is $r$-strong for some $r \in (0,1)$, by \Cref{prop:integralcondition}.

  \subsection*{$\tilde \om$ and $\tilde \si$ are weight functions} 
  It remains to show 
  $\tilde \om(2t) = O(\tilde \om(t))$ and $\tilde \si(2t) = O(\tilde \si(t))$ as $t \to \infty$.
  For $\tilde \om$ this follows from the concavity, but it is also a consequence of the following 
  argument which we present for $\tilde \si$. 
  Since $\si(2t) = O(\si(t))$ as $t \to \infty$, there is $N \in \N_{>2}$ and $C\ge 1$ such that 
  \begin{equation} \label{om1}
     \si(2t) \le C \si(t), \quad \text{ for all } t \ge x_N. 
   \end{equation} 
   Fix $t \ge x_N$.
   There is $n \in \N_{\ge N}$ such that $t \in [x_n,x_{n+1})$ and 
   consequently $2t  \in [x_n,x_{n+2})$, by \eqref{px1}. 
  Then  
  \begin{align*}
    \tilde \si(2t) \stackrel{\eqref{claim4}}{\le} (n+2) \si(2t) 
    \stackrel{\eqref{om1}}{\le} (n+2) C \si(t)
    \stackrel{\eqref{claim3}}{\le} \frac{n+2}{n-2} C \tilde \si(t)
    \le C' \tilde \si(t).
  \end{align*}
  Thus  $\tilde \si(2t) = O(\tilde \si(t))$ as $t \to \infty$.
  The proof is complete.
\end{proof}

\begin{remark}
	Note that $\om$ and $\tilde \om$ are non-quasianalytic which follows from \Cref{remark}(4) and \ref{i:3} and \ref{t:3}, respectively. 
	If $\si$ is non-quasianalytic, then also $\tilde \si$ can be chosen non-quasianalytic. 
	It suffices to demand 
	\[
	\int_{x_n}^\infty \frac{\si(t)}{1+t^2} dt \le  \frac{1}{n^3}
	\]
	in the construction of $(x_n)$.  Then 
	\begin{align*}
	    \int_{x_2}^\infty \frac{\tilde \si(t)}{1+t^2}\, dt 
 	   = \sum_{n=2}^\infty \int_{x_n}^{x_{n+1}} \frac{\tilde \si(t)}{1+t^2}\, dt 
  	  \stackrel{\eqref{claim4}}{\le} 
  	  \sum_{n=2}^\infty \int_{x_n}^{x_{n+1}} \frac{n \si(t)}{1+t^2}\, dt
  	  \le 
  	  \sum_{n=2}^\infty  \frac{1}{n^2}< \infty.
	\end{align*}

\end{remark}

\section{Proof of \Cref{Beurling}} \label{sec:proofBeurling}
The proof is based on \Cref{lemma2} and a well-known reduction scheme to the Roumieu case, i.e., \Cref{Roumieu}. 
We follow closely the arguments of \cite[Theorem 4.5]{BBMT91} and need two additional lemmas.

\begin{lemma}[{\cite[p.210]{BMT90}}] \label{C1weight}
  Each weight function $\om$ is equivalent to a weight function 
  $\bar \om$ of class $C^1$.
  If $\bar \vh = \bar \om \o \exp$, 
  then $\bar \vh'(t) \to \infty$ as $t \to \infty$. 
\end{lemma}

\begin{proof}
  The existence of $\bar \om$ was proved in \cite[p.210]{BMT90}. 
  Since $\bar \om$ is a weight function, $\bar \vh = \bar \om \o \exp$ is convex, $\bar \vh(0)=0$, and 
  $\bar \vh(t)/t \to \infty$ as $t\to \infty$. Convexity implies 
  $\bar \vh(t)/t \le \bar \vh'(t)$ for all $t>0$ and hence the statement follows. 
\end{proof}

\begin{lemma}[{\cite[Lemma 4.3]{BBMT91}}] \label{lem:BBMT43}
    Let $(C_j)_j$ be a positive sequence and 
    let $(\ps_j)_j$ be a sequence of differentiable functions on $[0,\infty)$ 
    such that for all $j$ the following conditions are satisfied:
    \begin{enumerate}[label=\textnormal{($\arabic*$)}]
      \item $\ps_j$ is convex, increasing, and $\ps_j(0)=0$.
      \item $\ps_j'(t) > \ps_{j+1}'(t)$ for all $t>0$. 
      \item $\lim_{t\to \infty} (\ps_j(t) - \ps_{j+1}(t)) = \infty$.
      \item $\lim_{t\to \infty} \ps_j'(t) = \infty$. 
    \end{enumerate}
    Then there exist a positive sequence $(D_j)_j$ and a convex function 
    $h : [0,\infty) \to [0,\infty)$ such that 
    \[
        \inf_{j \ge 1} (\ps_j(t) + C_j) \le h(t) \le  \inf_{j \ge 1} (\ps_j(t) + D_j)
        \quad \text{ for all } t>0. 
    \]
\end{lemma}

\begin{proof}[Proof of \Cref{Beurling}]
By \Cref{remark}(2) it suffices to show the extension result for compact sets. 
So let $K \subseteq \R^n$ be a non-empty compact set and fix a Whitney ultrajet 
$F=(F^\al) \in \cE^{(\si)}(K)$. Consider the sequences 
\begin{align*}
  a_k := \sup_{|\al| = k} \sup_{x \in E} |F^\al(x)|
\end{align*}
and 
\begin{align*}
  b_{k+1} := \sup_{|\al| \le k} \sup_{\substack{x,y \in K\\x \ne y}} |(R^k_x F)^\al(y)| 
  \frac{(k+1-|\al|)!}{|x-y|^{k+1-|\al|}}, \quad b_0:=0.
\end{align*}
Since $F$ is of class $\cE^{(\si)}$ and since the Young conjugate $\ps^*$ 
of $\ps : t \mapsto \si(e^t)$ is increasing, 
for each $j \in \N_{\ge 1}$ there exists $C_j> 1$ such that 
\[
  \max\{a_k,b_k\} \le C_j \exp(j \ps^*(k/j)) \quad \text{ for all } k\in \N. 
\]
Define the function $g : [0,\infty) \to \R$ by setting
\[
    g(t) := \log \max \{a_k,b_k,1\}, \quad \text{ for } k \le t < k+1. 
\]
Then for each $j \in \N_{\ge1}$ there exists $C_j>0$ such that
\[
   g(t) \le j \ps^*(t/j) + C_j \quad \text{ for all } t\ge 0. 
\] 

Let $\ps_j(t) := j \ps^*(t/j)$. We may assume without loss of generality that $\ps$ 
is of class $C^1$ and $\ps'(t) \to \infty$ as $t \to \infty$ (by \Cref{C1weight}), 
hence the Young conjugate $\ps^*$ is differentiable and satisfies 
$(\ps^*)' =(\ps')^{-1}$. 
By \Cref{lem:BBMT43} there exist a convex function $h : [0,\infty) \to [0,\infty)$ 
and a positive sequence $(D_j)$
such that
\begin{equation*}
  g \le h \le \inf_{j \ge 1} (\ps_j + D_j).
\end{equation*}
Then, for each $j \ge 1$ and each $t>0$,
\begin{equation*}
  h^*(t)= \sup_{s\ge 0} \big(st - h(s)\big) \ge  j\sup_{s\ge 0} \big(st/j - \ps^*(s/j)\big) - D_j = j \ps(t) -D_j.    
\end{equation*} 
The function $f(t) := h^*(\max\{0,\log(t)\})$ hence satisfies
\[
  \si(t) = \ps(\log(t)) \le \frac{1}{j} f(t) + \frac{D_j}{j}
\]
for all $j \ge 1$ and all $t \ge 1$.  
Thus $\si(t) = o(f(t))$ as $t \to \infty$.

By \Cref{lemma2}, there exist weight functions $\tilde \om$ and 
$\tilde \si$ with the properties \ref{t:1}--\ref{t:3} and such that  
\begin{equation} \label{eq:inclusions}
  \om(t) = o(\tilde \om(t)),\quad \si(t) = o(\tilde \si(t)),\quad  
  \tilde \si(t) = o(f(t)) \quad \text{ as } t \to \infty.
\end{equation}
By \Cref{lem:stronger}, 
the pair $(\tilde \om,\tilde \si)$ is strong and consequently satisfies the assumptions of \Cref{Roumieu}.

By \eqref{eq:inclusions}, there is a constant $B>0$ such hat $\tilde \si \le f + B$.
  Then we have 
  \[
    \tilde \ps(t) := \tilde \si(e^t) \le f(e^t) + B = h^*(t) + B \quad \text{ for all }
    t\ge 0 
  \]
  and furthermore (since $h$ is convex)
  \[
    g\le h = h^{**} \le \tilde \ps^* + B.
  \]
  This shows that the Whitney ultrajet $F$ belongs to $\cE^{\{\tilde \si\}}(K)$. 
  Then \Cref{Roumieu} implies that  
  $F \in j^\infty_K \cE^{\{\tilde \om\}}(\R^n)$. 
  Since, by \eqref{eq:inclusions}, $\om(t) = o(\tilde \om(t))$ as $t \to \infty$ and hence 
  $\cE^{\{\tilde \om\}}(\R^n) \subseteq \cE^{(\om)}(\R^n)$, \Cref{Beurling} is proved.
\end{proof}

\section{Extension operator} \label{sec:operator}

This section is dedicated to the proof of \Cref{operator}.
It is based on the rather explicit extension procedure available in the Roumieu case; see \cite{Rainer:2019ac,Rainer:2020aa}. 

Let us recall that due to \cite[Theorem 5.14]{MR3285413} we have 
\[
  \cE^{(\om)}(\R^n) = \on{proj}_{U \Subset \R^n} \on{proj}_{x>0} \on{proj}_{a>0} \cE^{W^x}_a(U)
\]
and 
\[
  \cE^{\{\om\}}(\R^n)  =  \on{proj}_{U \Subset \R^n} \on{ind}_{x>0} \on{ind}_{a>0} \cE^{W^x}_a(U)
\]
as locally convex spaces, where 
\[
\cE^{W^x}_a(U)= \Big\{f\in C^\infty(U) : \|f\|^{W^x}_{U,a} :=\sup_{x \in U} \sup_{\al \in \N^n} \frac{|f^{(\al)}(x)|}{a^{|\al|} W^x_{|\al|} } <\infty \Big\}
\]
is a Banach space associated with the weight sequence 
\[
    W^x_{k} := \exp\big( \tfrac{1}{x} \vh^* (k x)\big), \quad k \in \N.
\]
The family $\fW= \{W^x\}_{x>0}$ is called the \emph{weight matrix} associated with $\om$. 

If $K$ is a compact subset of $\R^n$ then we similarly have topological isomorphisms of the jet spaces 
\[
  \cE^{(\om)}(K) = \on{proj}_{x>0} \on{proj}_{a>0} \cE^{W^x}_a(K)
\]
and 
\[
  \cE^{\{\om\}}(K)  = \on{ind}_{x>0} \on{ind}_{a>0} \cE^{W^x}_a(K),
\]
where 
\[
\cE^{W^x}_a(K)= \Big\{F\in \cE(K) : \|F\|^{W^x}_{K,a} + |F|^{W^x}_{K,a}  <\infty \Big\}
\]
and 
\[
  | F |^{W^x}_{K,a}  := \sup_{\substack{x,y \in K\\ x \ne y}} \sup_{p \in \N } \sup_{|\al| \le p}  
  |(R^p_x F)^\al (y)|  \frac{(p+1-|\al|)!}{a^{p+1} W^x_{p+1} |x-y|^{p+1-|\al|} }. 
\]
These results are based on the following observation (see \cite[(5.10)]{MR3285413}): for any weight function $\om$ we have 
\begin{equation}
  \forall a>0 ~\exists H \ge 1 ~\forall x>0 ~\exists C\ge 1 ~\forall k \in \N :
  a^k W^x_k \le C W^{Hx}_k.
\end{equation}
It follows easily that the inclusions $\cE^{W^x}_a(\R^n) \to \cE^{\{\om\}}(\R^n)$ and 
$\cE^{(\om)}(K) \to \cE^{W^x}_a(K)$ are continuous for all $x,a>0$ (cf.\ \cite[Theorem 5.14.1]{MR3285413}).

\begin{proof}[Proof of \Cref{operator}]
	We apply \Cref{lemma2} to $f = \ta$. 
	Then there exist weight functions $\tilde\om$ and $\tilde \si$ satisfying \ref{t:1}--\ref{t:3}  as well as
	$\om(t) = o(\tilde \om(t))$, $\si(t) = o(\tilde \si(t))$, and 
  	$\tilde \si(t) = o(\ta(t))$ as $t \to \infty$.
	
	Let $K$ be a compact subset of $\R^n$.
	Now $\cE^{\{\ta\}}(K)$ is continuously included in $\cE^{(\tilde \si)}(K)$ thanks to $\tilde \si(t) = o(\ta(t))$ as $t \to \infty$,
	and  $\cE^{\{\tilde \om\}}(\R^n)$ is continuously included in $\cE^{(\om)}(\R^n)$ since $\om(t) = o(\tilde \om(t))$ as $t \to \infty$. 
	
	Let $\tilde \fS =\{\tilde S\}$ be the weight matrix associated with $\tilde \si$ and $\tilde \fW = \{\tilde W\}$ the one associated with $\tilde \om$.
	It was proved in \cite[Remark 5.6]{Rainer:2019ac}, \cite{Rainer:2020aa} that for each $\tilde S \in \tilde \fS$ and each $a>0$ there exist 
  	$\tilde W \in \tilde \fW$, $b>0$, and a continuous linear 
	extension operator $\cE^{\tilde S}_a(K) \to \cE^{\tilde W}_b(\R^n)$.   
	Since also the natural inclusions $\cE^{(\tilde \si)}(K) \to \cE^{\tilde S}_a(K)$ and $\cE^{\tilde W}_b (\R^n) \to \cE^{\{\tilde \om\}}(\R^n)$ are continuous, 
    there exists a continuous linear extension operator $\cE^{\{\ta\}}(K) \to \cE^{(\om)}(\R^n)$.

  \[
  \xymatrix{
  \cE^{(\si)}(K) &\cE^{\{\ta\}}(K) \ar@{^{ (}->}[d] \ar@{_{ (}->}[l] \ar@{-->}[rrr] &&&  \cE^{(\om)}(\R^n)
  \\
  &\cE^{(\tilde \si)}(K) \ar@{^{ (}->}[r] \ar@{_{ (}->}[ul] &\cE^{\tilde S}_a(K) \ar[r] & \cE^{\tilde W}_b(\R^n) \ar@{^{ (}->}[r] & \cE^{\{\tilde \om\}}(\R^n) 
  \ar@{^{ (}->}[u]
  }
  \]

  If $A$ is a closed subset of $\R^n$, then we may use a suitable partition of unity as in \Cref{remark}(2) 
  in order to obtain the required extension operator.
\end{proof}


\def\cprime{$'$}
\providecommand{\bysame}{\leavevmode\hbox to3em{\hrulefill}\thinspace}
\providecommand{\MR}{\relax\ifhmode\unskip\space\fi MR }
\providecommand{\MRhref}[2]{%
  \href{http://www.ams.org/mathscinet-getitem?mr=#1}{#2}
}
\providecommand{\href}[2]{#2}

\end{document}